\documentclass[11pt]{article}

\usepackage[T1]{fontenc}
\usepackage{lmodern}
\usepackage{microtype}
\usepackage[letterpaper,margin=1in]{geometry}
\usepackage{amsmath,amssymb,amsthm,mathtools,mathrsfs}
\usepackage{enumitem}
\usepackage[hidelinks]{hyperref}

\numberwithin{equation}{section}

\newtheorem{theorem}{Theorem}[section]
\newtheorem{proposition}[theorem]{Proposition}
\newtheorem{lemma}[theorem]{Lemma}
\newtheorem{corollary}[theorem]{Corollary}
\newtheorem{remark}[theorem]{Remark}
\newtheorem{definition}[theorem]{Definition}

\newcommand{\C}{\mathbb C}
\newcommand{\R}{\mathbb R}
\newcommand{\cR}{\mathcal R}
\newcommand{\cI}{\mathcal I}
\newcommand{\CT}{\operatorname{CT}}
\newcommand{\Int}{\operatorname{Int}}
\newcommand{\Aut}{\operatorname{Aut}}
\newcommand{\Ad}{\operatorname{Ad}}
\newcommand{\ad}{\operatorname{ad}}
\newcommand{\Span}{\operatorname{span}}
\newcommand{\dd}{\,\mathrm d}
\newcommand{\ip}[2]{\langle #1,#2\rangle}

\title{The Mathieu Property for Compact Connected Lie Groups}
\author{Christopher D. Long\\[2pt]
\small\texttt{galizur@gmail.com}}
\date{July 30, 2026}

\newcommand{\mscclassification}{22E30, 22E46, 43A75, 16D70}
\newcommand{\paperkeywords}{Mathieu conjecture, Mathieu--Zhao subspace, compact Lie group,
Haar measure, root subgroup, Hopf coordinates, Jacobian conjecture}

\hypersetup{
  pdftitle={The Mathieu Property for Compact Connected Lie Groups},
  pdfauthor={Christopher D. Long},
  pdfsubject={Classification of the Mathieu property for compact connected Lie groups},
  pdfkeywords={Mathieu conjecture, Mathieu--Zhao subspace, compact Lie group, Haar measure, root subgroup, Hopf coordinates, Jacobian conjecture}
}

\begin{document}

\maketitle

\begin{abstract}
Let $G$ be a compact connected Lie group, let $\cR(G)$ denote its algebra of
representative functions, and let $\cI_G(f)=\int_G f(g)\,\dd g$ be normalized
Haar integration.  We prove that $\ker \cI_G$ is a Mathieu--Zhao subspace of
$\cR(G)$ if and only if $G$ is a torus.  More strongly, for every nonabelian
compact connected Lie group $G$ we construct $A,P,Q\in\cR(G)$, with $A\geq0$
and $A\not\equiv0$, such that all pure moments of $P$ vanish and the marked
moments satisfy an exact Pascal-row identity
\[
 \cI_G(Q^sP^m)
 =c_m\binom{m-1}{s-1}\cI_G(A^{4m+s})>0
 \qquad(1\leq s\leq m),
\]
where $c_m=4^m(m!)^2/(2m+1)!$; the same moments vanish for $s>m$.
The construction begins with a homogeneous, phase-balanced polynomial pair on
$\C^2$.  An exact coefficient identity on the Hopf sphere $S^3$ is transferred
by orbit averaging to every compactly supported $SU(2)$-invariant measure on
$\C^2$.  A highest-weight representation and a visible simple-root doublet
produce such a measure on every compact simple Lie group.  Phase balance gives
descent through every central quotient, and pullback from an adjoint simple
quotient handles arbitrary nonabelian compact connected groups.  The proof is
direct and independent of the implication from the Mathieu conjecture to the
Jacobian conjecture.  The torus direction is the theorem of Duistermaat and
van der Kallen.
\end{abstract}

\begin{center}
\begin{minipage}{0.92\textwidth}
\small
\noindent\textit{2020 Mathematics Subject Classification.} \mscclassification

\end{minipage}
\end{center}

\section{Introduction}

Mathieu formulated in \cite{Mathieu1997} a moment conjecture for finite-type
functions on compact connected Lie groups.  In the language of Mathieu
subspaces introduced and developed by Zhao \cite{Zhao2012}, it asks whether
the kernel of normalized Haar integration is a Mathieu subspace of the
algebra of representative functions.

Duistermaat and van der Kallen proved the conjecture for tori
\cite{DuistermaatVanderKallen1998}.  Dings and Koelink studied the first
nonabelian case through matrix coefficients \cite{DingsKoelink2015}.
M\"uger--Tuset and Zwart subsequently used explicit Euler- and KAK-type
integration formulas to reduce the conjecture for $SU(2)$, $SU(N)$, $SO(N)$,
$Sp(N)$, $G_2$, and ultimately general compact connected Lie groups to moment
or convex-support conjectures on abelian parameter spaces
\cite{MugerTuset2024,MugerTuset2025,Zwart2023,Zwart2024,Zwart2025}.

There are now two independent sources of counterexamples.  First, an explicit
counterexample for $SU(2)$ was given in \cite{Long2026}.  Second, in July 2026
Alp\"oge announced an explicit three-dimensional counterexample to the
Jacobian conjecture, credited in the announcement to Fable
\cite{Alpoge2026}.  The announced map has constant Jacobian determinant $-2$
and is generically three-to-one \cite{SecretBloggingSeminar2026}.  Xena
reported that Lezeau had manually formalized the example and opened a pull
request to the Formal Conjectures repository \cite{Xena2026}; that pull
request was merged on July 26, 2026 \cite{Lezeau2026}.  Postcomposing with the
diagonal target automorphism $\operatorname{diag}(-\tfrac12,1,1)$ normalizes
the Jacobian determinant to $1$ while preserving noninjectivity.  Adjoining
$N-3$ identity coordinates then gives a counterexample in every dimension
$N\geq3$.  Since the Mathieu conjecture for $SU(N)$ implies the Jacobian
conjecture in dimension $N$ \cite{Mathieu1997,ZwartMathieu2026}, it follows
that the Mathieu conjecture fails for $SU(N)$ for every $N\geq3$.  These
results do not by themselves classify arbitrary compact connected groups.
The purpose of the present paper is to give a direct, uniform classification,
independent of the Jacobian route and without case-by-case analysis by Lie
type.

Let $\cR(G)$ be the algebra of representative functions on a compact group
$G$, equivalently the finite linear combinations of matrix coefficients of
finite-dimensional continuous complex representations.  Let
\[
 \cI_G:\cR(G)\longrightarrow\C,
 \qquad
 \cI_G(f)=\int_G f(g)\,\dd g,
\]
where Haar measure is normalized to have total mass one.  We say that $G$ has
the \emph{Mathieu property} if $\ker\cI_G$ is a Mathieu subspace of
$\cR(G)$.

Our main result is the following complete classification.

\begin{theorem}[Classification theorem]\label{thm:classification}
Let $G$ be a compact connected Lie group.  Then the following conditions are
equivalent:
\begin{enumerate}[label=\textup{(\roman*)}]
\item $\ker\cI_G$ is a Mathieu--Zhao subspace of $\cR(G)$;
\item $G$ is abelian;
\item $G$ is a torus.
\end{enumerate}
\end{theorem}

The nonabelian direction is substantially stronger than the negation of
eventual vanishing.  We use the convention
$\binom{r}{k}=0$ when $k<0$ or $k>r$.

\begin{theorem}[Uniform nonabelian marker tower]\label{thm:uniform-nonabelian}
For every nonabelian compact connected Lie group $G$, there exist
$A,P,Q\in\cR(G)$ such that $A(g)\geq0$ for every $g\in G$, $A\not\equiv0$, and,
for every integer $m\geq1$,
\begin{equation}\label{eq:uniform-pure}
 \cI_G(P^m)=0.
\end{equation}
For every integer $s\geq1$ one has
\begin{equation}\label{eq:uniform-marker-tower}
 \cI_G(Q^sP^m)
 =c_m\binom{m-1}{s-1}\cI_G(A^{4m+s}),
 \qquad
 c_m=\frac{4^m(m!)^2}{(2m+1)!}>0.
\end{equation}
Consequently, the marked moment is strictly positive for $1\leq s\leq m$ and
vanishes for $s>m$.  In particular, the same fixed multiplier $Q$ detects
every positive power of $P$.
\end{theorem}

The proof has four steps.  First, we establish an exact coefficient formula
for a homogeneous, phase-balanced polynomial pair on the Hopf sphere $S^3$.
Second, we transfer the full coefficient identity by orbit averaging to
arbitrary compactly supported $SU(2)$-invariant measures on $\C^2$.  Third, a
fundamental representation and a highest-weight-one simple-root doublet
produce such a measure on every simply connected compact simple group; phase
balance makes the construction descend through every central quotient.
Finally, every nonabelian compact connected group has an adjoint compact
simple quotient.
For the defining representations of $SU(n)$ and $Sp(n)$, Proposition
\ref{prop:classical-closed-forms} evaluates the entire marker tower in closed
beta-function form.

The standard facts concerning representative functions, compact Lie groups,
root subgroups, highest weights, and the associated $\mathfrak{sl}_2$-modules
used below may be found in \cite{BrockerTomDieck1985,Hall2015}.  All moment
calculations and all transfer arguments specific to the present construction
are proved explicitly.

\section{Mathieu subspaces and functoriality}

\begin{definition}\label{def:mathieu-subspace}
Let $A$ be a commutative unital $\C$-algebra.  A linear subspace $M\subseteq A$
is a \emph{Mathieu subspace} if, whenever $f^m\in M$ for every $m\geq1$, one
has
\[
 hf^m\in M
 \qquad\text{for every fixed }h\in A\text{ and all sufficiently large }m.
\]
For a compact group $G$, we say that $G$ has the Mathieu property if
$\ker\cI_G$ is a Mathieu subspace of $\cR(G)$.
\end{definition}

For compact groups, representative functions are precisely the finite-type
functions appearing in Mathieu's original formulation.  We record the
closure and functoriality properties used later.

\begin{lemma}[Representative functions]\label{lem:representative-algebra}
Let $G$ be a compact group.
\begin{enumerate}[label=\textup{(\roman*)}]
\item $\cR(G)$ is a unital $*$-subalgebra of $C(G)$.
\item If $f$ and $h$ are matrix coefficients of representations $\rho$ and
$\sigma$, respectively, then $fh$ is a matrix coefficient of
$\rho\otimes\sigma$, and $\overline f$ is a matrix coefficient of the
conjugate representation $\overline\rho$.
\item If $\pi:G\to H$ is a continuous homomorphism of compact groups, then
$f\circ\pi\in\cR(G)$ for every $f\in\cR(H)$.
\end{enumerate}
\end{lemma}

\begin{proof}
The constant function $1$ is the coefficient of the trivial representation.
If
$f(g)=\ip{\rho(g)v}{\lambda}$ and
$h(g)=\ip{\sigma(g)w}{\mu}$, then
\[
 f(g)h(g)
 =\ip{(\rho\otimes\sigma)(g)(v\otimes w)}{\lambda\otimes\mu}.
\]
Complex conjugation gives a coefficient of $\overline\rho$.  Finally,
$f\circ\pi$ is a coefficient of $\rho\circ\pi$ whenever $f$ is a coefficient
of $\rho$.  Taking finite linear combinations proves the assertions.
\end{proof}

\begin{lemma}[Haar pushforward and pullback]\label{lem:haar-pullback}
Let $\pi:G\twoheadrightarrow H$ be a continuous surjective homomorphism of
compact groups.  Then the pushforward of normalized Haar measure on $G$ is
normalized Haar measure on $H$.  Consequently,
\begin{equation}\label{eq:haar-pullback}
 \cI_G(f\circ\pi)=\cI_H(f)
 \qquad(f\in C(H)).
\end{equation}
In particular, a counterexample to the Mathieu property on $H$ pulls back to
a counterexample on $G$.
\end{lemma}

\begin{proof}
The pushforward measure is a left-invariant probability measure on $H$, hence
is normalized Haar measure by uniqueness.  Equation
\eqref{eq:haar-pullback} follows.  Lemma \ref{lem:representative-algebra}
shows that representative functions pull back to representative functions.
\end{proof}

\section{The universal Hopf pair}

For a polynomial or formal power series $F(X)$, we write $[X^m]F(X)$ for
the coefficient of $X^m$; for a Laurent polynomial in $u$, we write
$\CT_u$ for its constant term.

For $z=(z_0,z_1)\in\C^2$, define the phase-balanced quadratic polynomials
\begin{equation}\label{eq:hopf-coordinates}
 \begin{alignedat}{2}
 a&=|z_0|^2+|z_1|^2,
 &\qquad \tau&=|z_0|^2-|z_1|^2,\\
 u&=2z_0\overline{z_1},
 &v&=2z_1\overline{z_0}.
 \end{alignedat}
\end{equation}
They satisfy the Hopf relation
\begin{equation}\label{eq:hopf-relation}
 a^2=uv+\tau^2.
\end{equation}
Define
\begin{equation}\label{eq:universal-pair}
 \boxed{
 \mathscr P=(a+u)\bigl(a^2v-(2a+u)\tau^2\bigr),
 \qquad
 \mathscr Q=u.}
\end{equation}
Both are polynomials in $z_0,z_1,\overline{z_0},\overline{z_1}$.  They are
invariant under the common phase action
$(z_0,z_1)\mapsto(\lambda z_0,\lambda z_1)$ for $|\lambda|=1$, and their real
homogeneities are
\begin{equation}\label{eq:homogeneity}
 \mathscr P(rz)=r^8\mathscr P(z),
 \qquad
 \mathscr Q(rz)=r^2\mathscr Q(z)
 \qquad(r\geq0).
\end{equation}

The two structural constraints serve different later steps.  Homogeneity makes
each orbit-sphere moment a single radial power and is essential to the transfer
argument in Section~\ref{sec:radial-transfer}; phase balance makes all
quantities invariant under scalar central characters and is essential to the
descent argument in Lemma~\ref{lem:center-descent}.  Thus an inhomogeneous
counterexample on the unit sphere need not survive this mechanism.

Let $S^3=\{z\in\C^2:a(z)=1\}$ and let $\sigma$ be normalized surface measure.
Write
\[
 p=\mathscr P|_{S^3},
 \qquad
 q=\mathscr Q|_{S^3}.
\]
For $m\geq1$, set
\begin{equation}\label{eq:cm-definition}
 c_m=\int_0^1(1-s^2)^m\,\dd s
 =\frac12 B\!\left(\frac12,m+1\right)
 =\frac{4^m(m!)^2}{(2m+1)!}
 =\frac{2^m m!}{(2m+1)!!}.
\end{equation}
Here $B$ denotes Euler's beta function.  In particular, $c_m>0$.

\begin{theorem}[Hopf coefficient identity]\label{thm:hopf-coefficient}
For every integer $m\geq1$ and every polynomial $H\in\C[X]$,
\begin{equation}\label{eq:hopf-coefficient}
 \int_{S^3}H(q(z))p(z)^m\,\dd\sigma(z)
 =c_m[X^m]\,H(X)(1+X)^{m-1}.
\end{equation}
\end{theorem}

\begin{proof}
Use Hopf coordinates
\begin{equation}\label{eq:hopf-parametrization}
 z_0=\sqrt{\frac{1+t}{2}}\,e^{i\alpha},
 \qquad
 z_1=\sqrt{\frac{1-t}{2}}\,e^{i\beta},
\end{equation}
where $-1\leq t\leq1$ and $0\leq\alpha,\beta<2\pi$.  Outside the two endpoint
circles of measure zero, these coordinates give
\begin{equation}\label{eq:s3-measure}
 \dd\sigma=\frac{1}{8\pi^2}\,\dd t\,\dd\alpha\,\dd\beta,
 \qquad
 \tau=t,
 \qquad
 u=\sqrt{1-t^2}\,e^{i(\alpha-\beta)}.
\end{equation}
On $S^3$, relation \eqref{eq:hopf-relation} becomes $uv+\tau^2=1$.
Multiplying the definition of $p$ by $u$ gives the polynomial identity
\begin{equation}\label{eq:no-division-identity}
 up=(1+u)\bigl(1-(1+u)^2\tau^2\bigr).
\end{equation}
Thus, away from the measure-zero set $u=0$,
\begin{equation}\label{eq:rational-p}
 p=u^{-1}(1+u)\bigl(1-(1+u)^2t^2\bigr).
\end{equation}
The preceding polynomial identity is the global justification for this
localization.  In particular, $p$ is the restriction of a polynomial to the
compact sphere and is bounded; the apparent pole in \eqref{eq:rational-p} is
used only off the null set $u=0$ and creates no integrability issue.  Thus all
integrals below are absolutely convergent.

For fixed $-1<t<1$, write $u=r_t e^{i\theta}$ with
$r_t=\sqrt{1-t^2}$.  Averaging over the phase
$\theta=\alpha-\beta$ extracts the Laurent constant term in $u$: if
$F(u)=\sum_j b_j(t)u^j$, then
$(2\pi)^{-1}\int_0^{2\pi}F(r_t e^{i\theta})\,\dd\theta=b_0(t)$.
The $(\alpha,\beta)$-integration contributes one half of this normalized phase
average.  The resulting constant term is even in $t$, so the factor $1/2$
cancels the factor $2$ obtained by restricting from $[-1,1]$ to $[0,1]$.
Consequently,
\begin{align}
 \int_{S^3}H(q)p^m\,\dd\sigma
 &=\frac12\int_{-1}^1\CT_u\!\left(
 H(u)u^{-m}(1+u)^m
 \bigl(1-(1+u)^2t^2\bigr)^m\right)\dd t\notag\\
 &=\int_0^1\CT_u\!\left(
 H(u)u^{-m}(1+u)^m
 \bigl(1-(1+u)^2t^2\bigr)^m\right)\dd t\notag\\
 &=[X^m]H(X)(1+X)^m
 \int_0^1\bigl(1-(1+X)^2t^2\bigr)^m\,\dd t.
 \label{eq:ct-step}
\end{align}
Define the odd polynomial
\[
 J_m(y)=\int_0^y(1-s^2)^m\,\dd s.
\]
The substitution $s=(1+X)t$ in \eqref{eq:ct-step} gives
\begin{equation}\label{eq:J-step}
 \int_{S^3}H(q)p^m\,\dd\sigma
 =[X^m]H(X)(1+X)^{m-1}J_m(1+X).
\end{equation}
Moreover,
\begin{align}
 J_m(1+X)-J_m(1)
 &=\int_0^X\bigl(1-(1+y)^2\bigr)^m\,\dd y\notag\\
 &=\int_0^X\bigl(-y(2+y)\bigr)^m\,\dd y
 \in X^{m+1}\C[X].
 \label{eq:J-congruence}
\end{align}
Since $H(X)(1+X)^{m-1}$ has no negative powers, the term in
$X^{m+1}\C[X]$ contributes nothing to the coefficient of $X^m$ in
\eqref{eq:J-step}.  Hence $J_m(1+X)$ may be replaced by
$J_m(1)=c_m$, proving \eqref{eq:hopf-coefficient}.
\end{proof}

\begin{corollary}[Pure moments and Pascal marker tower]
\label{cor:sphere-marker-tower}
For every $m\geq1$,
\begin{equation}\label{eq:sphere-pure}
 \int_{S^3}p^m\,\dd\sigma=0.
\end{equation}
For every integer $s\geq1$,
\begin{equation}\label{eq:sphere-marked}
 \int_{S^3}q^sp^m\,\dd\sigma
 =c_m\binom{m-1}{s-1}.
\end{equation}
Thus the marked moments are exactly row $m-1$ of Pascal's triangle: they are
positive for $1\leq s\leq m$ and vanish for $s>m$.  In particular,
\begin{equation}\label{eq:sphere-first-marker}
 \int_{S^3}qp^m\,\dd\sigma=c_m>0.
\end{equation}
\end{corollary}

\begin{proof}
Set $H=1$ in Theorem \ref{thm:hopf-coefficient}.  Since
$(1+X)^{m-1}$ has degree $m-1$, the coefficient of $X^m$ is zero.  For
$H(X)=X^s$, the required coefficient is
\[
 [X^{m-s}](1+X)^{m-1}=\binom{m-1}{s-1},
\]
with the stated zero convention when $s>m$.
\end{proof}

\begin{remark}[Defect-one coefficient law]\label{rem:defect-one}
Equation \eqref{eq:hopf-coefficient} is a one-step boundary law for a Pascal
row.  The unmarked extraction lies one degree beyond
$(1+X)^{m-1}$ and vanishes, while multiplication by $q^s$ moves the extraction
to the coefficient of $X^{m-s}$.  No additional finite-difference identity is
needed in the proof.
\end{remark}

\section{Transfer to radial measures}\label{sec:radial-transfer}

We identify $SU(2)$ with its defining unitary action on $\C^2$.  This action is
transitive on every sphere centered at the origin: for a unit vector
$\zeta=(\zeta_0,\zeta_1)$, the matrix
\[
 \begin{pmatrix}\zeta_0&-\overline{\zeta_1}\\
 \zeta_1&\overline{\zeta_0}\end{pmatrix}\in SU(2)
\]
sends $(1,0)^{\mathsf T}$ to $\zeta$.

\begin{theorem}[Universal radial transfer]\label{thm:radial-transfer}
Let $\mu$ be a compactly supported finite positive Borel measure on $\C^2$
that is invariant under the defining action of $SU(2)$.  Then, for every
$m\geq1$,
\begin{equation}\label{eq:radial-pure}
 \int_{\C^2}\mathscr P(z)^m\,\dd\mu(z)=0.
\end{equation}
For every integer $s\geq1$,
\begin{equation}\label{eq:radial-marked}
 \int_{\C^2}\mathscr Q(z)^s\mathscr P(z)^m\,\dd\mu(z)
 =c_m\binom{m-1}{s-1}
 \int_{\C^2}a(z)^{4m+s}\,\dd\mu(z).
\end{equation}
Thus the marked moment vanishes for $s>m$.  If
$\mu(\C^2\setminus\{0\})>0$, then
\[
 \int_{\C^2}a(z)^{4m+s}\,\dd\mu(z)>0
 \qquad(m,s\geq1),
\]
and consequently
\begin{equation}\label{eq:radial-positive}
 \int_{\C^2}\mathscr Q(z)^s\mathscr P(z)^m\,\dd\mu(z)>0
 \qquad(m\geq1,\ 1\leq s\leq m).
\end{equation}
\end{theorem}

\begin{proof}
For any $\mu$-integrable function $F$, invariance of $\mu$ and Fubini's
theorem give
\begin{equation}\label{eq:orbit-average}
 \int_{\C^2}F(z)\,\dd\mu(z)
 =\int_{\C^2}\int_{SU(2)}F(kz)\,\dd k\,\dd\mu(z).
\end{equation}
For fixed $z\neq0$, write $z=r\zeta$ with $r=\|z\|$ and $\zeta\in S^3$.
The pushforward of Haar measure under $k\mapsto k\zeta$ is normalized surface
measure on $S^3$.  By \eqref{eq:homogeneity} and
Corollary \ref{cor:sphere-marker-tower},
\[
 \int_{SU(2)}\mathscr P(kz)^m\,\dd k
 =r^{8m}\int_{S^3}p^m\,\dd\sigma=0,
\]
and
\begin{align*}
 \int_{SU(2)}\mathscr Q(kz)^s\mathscr P(kz)^m\,\dd k
 &=r^{8m+2s}c_m\binom{m-1}{s-1}\\
 &=a(z)^{4m+s}c_m\binom{m-1}{s-1}.
\end{align*}
The same formulas hold at $z=0$.  Substitution into
\eqref{eq:orbit-average} proves \eqref{eq:radial-pure} and
\eqref{eq:radial-marked}.  If $\mu$ is not concentrated at the origin, then
$a^{4m+s}$ is nonnegative and positive on a set of positive $\mu$-measure for
every $m,s\geq1$.  For $1\leq s\leq m$, combine this with $c_m>0$ and
$\binom{m-1}{s-1}>0$ to obtain \eqref{eq:radial-positive}.
\end{proof}

\section{Root-\texorpdfstring{$SU(2)$}{SU(2)} transfer to compact simple groups}

We first treat a compact connected simply connected simple Lie group
$\widetilde G$.  Fix a maximal torus and a system of simple roots.  For a
simple root $\alpha$, let $\varpi_\alpha$ be the corresponding fundamental
weight.  Since $\widetilde G$ is simply connected, the highest-weight theorem
provides an irreducible representation
\begin{equation}\label{eq:fundamental-representation}
 \rho:\widetilde G\longrightarrow U(V)
\end{equation}
with highest weight $\varpi_\alpha$.  We choose a $\widetilde G$-invariant
Hermitian inner product on $V$, linear in the first variable.

\begin{lemma}[Visible root doublet]\label{lem:root-doublet}
Let $e_0\in V$ be a unit highest-weight vector.  There exist
\begin{enumerate}[label=\textup{(\roman*)}]
\item a root subgroup $K_\alpha\subseteq\widetilde G$ isomorphic to $SU(2)$;
\item a $K_\alpha$-invariant two-dimensional subspace
$W=\Span_\C\{e_0,e_1\}\subseteq V$;
\end{enumerate}
such that the representation of $K_\alpha$ on $W$ is unitarily equivalent to
the defining representation of $SU(2)$ on $\C^2$.
\end{lemma}

\begin{proof}
The defining property of the fundamental weight is
\[
 \ip{\varpi_\alpha}{\alpha^\vee}=1,
\]
where the brackets here denote the canonical weight--coroot pairing.
Restrict the representation to the $\mathfrak{sl}_2$-subalgebra associated
with $\alpha$.  The vector $e_0$ is an $\mathfrak{sl}_2$-highest-weight vector
of highest weight $1$.  Hence the cyclic $\mathfrak{sl}_2$-submodule generated
by $e_0$ is the irreducible highest-weight-one module.  If $F_\alpha$ is a
nonzero lowering operator, then
\[
 F_\alpha e_0\neq0,
 \qquad
 F_\alpha^2e_0=0,
\]
and the two weights are $\varpi_\alpha$ and
$\varpi_\alpha-\alpha$.  Let $e_1$ be a normalization of $F_\alpha e_0$ and
set $W=\Span_\C\{e_0,e_1\}$.

The compact real form of the root $\mathfrak{sl}_2$ integrates to a
homomorphism $\phi_\alpha:SU(2)\to\widetilde G$.  On $W$, the
representation $\rho\circ\phi_\alpha$ is the irreducible
highest-weight-one representation of $SU(2)$, hence the defining
two-dimensional representation.  This representation is faithful, so
$\phi_\alpha$ is injective.  Thus
$K_\alpha=\phi_\alpha(SU(2))\simeq SU(2)$, and the image of $K_\alpha$ in
$U(W)$ is the full special unitary group $SU(W)$.  In particular, in every
orthonormal coordinate system on $W$ this image is the standard subgroup
$SU(2)\subset U(2)$; no basis-invariance of the polynomial pair itself is
required.  These standard
root-subgroup and highest-weight facts are treated in
\cite[Chapters V--VI]{BrockerTomDieck1985} and
\cite[Parts II--III]{Hall2015}.
\end{proof}

Let $\Pi_W:V\to W$ be orthogonal projection.  Since $\rho$ is unitary and
$W$ is $K_\alpha$-invariant, $W^\perp$ is also $K_\alpha$-invariant, and
therefore
\begin{equation}\label{eq:projection-commutes}
 \Pi_W\rho(k)=\rho(k)\Pi_W
 \qquad(k\in K_\alpha).
\end{equation}
Define
\begin{equation}\label{eq:Phi-definition}
 \Phi:\widetilde G\longrightarrow W,
 \qquad
 \Phi(g)=\Pi_W\rho(g)e_0.
\end{equation}
Using the orthonormal basis $(e_0,e_1)$, identify $W$ unitarily with
$\C^2$.  By Lemma \ref{lem:root-doublet} and the preceding observation, the
image of $K_\alpha$ is the standard $SU(2)$ in these coordinates.  Write
\begin{equation}\label{eq:z-coordinates}
 \Phi(g)=z_0(g)e_0+z_1(g)e_1,
 \qquad
 z_j(g)=\ip{\rho(g)e_0}{e_j}.
\end{equation}

\begin{lemma}[Radial Haar pushforward]\label{lem:radial-pushforward}
The pushforward $\mu=\Phi_*(\dd g)$ of normalized Haar measure on
$\widetilde G$ is a compactly supported $SU(2)$-invariant probability measure
on $\C^2$, and it is not concentrated at the origin.
\end{lemma}

\begin{proof}
For $k\in K_\alpha$, equations \eqref{eq:projection-commutes} and
\eqref{eq:Phi-definition} give
\begin{equation}\label{eq:Phi-equivariance}
 \Phi(kg)=\rho(k)\Phi(g).
\end{equation}
Left invariance of Haar measure therefore implies invariance of $\mu$ under
the defining $SU(2)$-action on $W\simeq\C^2$.

Since $\rho(g)e_0$ is a unit vector and $\Pi_W$ is an orthogonal projection,
$\|\Phi(g)\|\leq1$; hence $\mu$ is supported in the closed unit ball.  At the
identity, $\Phi(1)=e_0$.  By continuity, $\|\Phi(g)\|>1/2$ on some nonempty
open neighborhood of the identity.  Every nonempty open subset of a compact
group has positive Haar measure, so $\mu(\C^2\setminus\{0\})>0$.
\end{proof}

Apply \eqref{eq:hopf-coordinates} and \eqref{eq:universal-pair} to the
coordinate pair $(z_0(g),z_1(g))$.  This defines functions
$A,u,v,\tau,P,Q$ on $\widetilde G$:
\begin{equation}\label{eq:group-pair}
 A=a\circ\Phi=|z_0|^2+|z_1|^2,
 \qquad
 P=\mathscr P\circ\Phi,
 \qquad
 Q=\mathscr Q\circ\Phi.
\end{equation}

\begin{theorem}[Simply connected simple groups]\label{thm:simply-connected-simple}
Let $\widetilde G$ be a compact connected simply connected simple Lie group.
The functions $A,P,Q$ in \eqref{eq:group-pair} belong to $\cR(\widetilde G)$,
with $A\geq0$ and $A\not\equiv0$, and satisfy, for every $m\geq1$,
\begin{equation}\label{eq:simple-pure}
 \int_{\widetilde G}P(g)^m\,\dd g=0.
\end{equation}
For every integer $s\geq1$,
\begin{equation}\label{eq:simple-marker-tower}
 \int_{\widetilde G}Q(g)^sP(g)^m\,\dd g
 =c_m\binom{m-1}{s-1}
 \int_{\widetilde G}A(g)^{4m+s}\,\dd g.
\end{equation}
The right-hand side is strictly positive for $1\leq s\leq m$ and is zero for
$s>m$.
\end{theorem}

\begin{proof}
Each $z_j$ in \eqref{eq:z-coordinates} is a matrix coefficient of $\rho$.
By Lemma \ref{lem:representative-algebra}, its conjugate and every polynomial
in the $z_j$ and $\overline{z_j}$ are representative functions.  Hence
$A,P,Q\in\cR(\widetilde G)$.  Moreover, $A=\|\Phi\|^2\geq0$, and
$A(1)=1$.  Lemma \ref{lem:radial-pushforward} allows us to apply Theorem
\ref{thm:radial-transfer} to $\mu=\Phi_*(\dd g)$, which gives all displayed
identities and the strict positivity.
\end{proof}

\begin{proposition}[Closed forms in the defining representations]
\label{prop:classical-closed-forms}
Choose the defining representation and the first simple root for $SU(n)$ or
$Sp(n)$.  Write $(x)_k=x(x+1)\cdots(x+k-1)$ for the rising factorial.  For
$SU(n)$ with $n\geq2$,
\begin{equation}\label{eq:SU-n-closed-form}
 \cI_{SU(n)}(Q^sP^m)
 =c_m\binom{m-1}{s-1}\frac{(2)_{4m+s}}{(n)_{4m+s}}
 \qquad(s\geq1).
\end{equation}
For $Sp(n)$ with $n\geq1$,
\begin{equation}\label{eq:Sp-n-closed-form}
 \cI_{Sp(n)}(Q^sP^m)
 =c_m\binom{m-1}{s-1}\frac{(2)_{4m+s}}{(2n)_{4m+s}}
 \qquad(s\geq1).
\end{equation}
In particular, the first marked
moments for $m=1,2,3$ are
\[
\begin{array}{c|ccc}
 G & m=1 & m=2 & m=3 \\ \hline
 SU(2) & \dfrac23 & \dfrac8{15} & \dfrac{16}{35} \\[3pt]
 SU(3) & \dfrac4{21} & \dfrac{16}{165} & \dfrac{32}{525} \\[3pt]
 SU(4),\ Sp(2) & \dfrac1{14} & \dfrac4{165} & \dfrac2{175}
\end{array}
\]
\end{proposition}

\begin{proof}
For the defining representation of $SU(n)$, the map $\rho(g)e_0$ is the first
column of $g$, uniformly distributed on the unit sphere in $\C^n$, and
$A=|g_{11}|^2+|g_{21}|^2$.  If $n\geq3$, the vector of squared coordinate
moduli has the Dirichlet distribution with all parameters equal to $1$;
this follows, for example, by normalizing independent standard complex
Gaussian variables.  Hence
\[
 A\sim\operatorname{Beta}(2,n-2),
 \qquad
 \cI_{SU(n)}(A^k)=\frac{B(k+2,n-2)}{B(2,n-2)}
 =\frac{(2)_k}{(n)_k}.
\]
At the endpoint $n=2$, one has $A\equiv1$, and the same rising-factorial
formula remains valid because $(2)_k/(2)_k=1$.  Substitution of $k=4m+s$ into
\eqref{eq:simple-marker-tower} proves \eqref{eq:SU-n-closed-form}.

The defining action of $Sp(n)$ is transitive on the unit sphere in
$\C^{2n}$, and its invariant probability measure there is normalized surface
measure.  For $n\geq2$, the same argument gives
$A\sim\operatorname{Beta}(2,2n-2)$.  At the endpoint $n=1$, one has
$Sp(1)=SU(2)$ and $A\equiv1$, while $(2)_k/(2n)_k=1$.  This proves
\eqref{eq:Sp-n-closed-form}.  The displayed values follow by direct
simplification.
\end{proof}

We now verify descent through the full center, including finite-type status on
the quotient.

\begin{lemma}[Center descent]\label{lem:center-descent}
Let $Z=Z(\widetilde G)$.  The functions $A,u,v,\tau,P,Q$ above are invariant
under left multiplication by $Z$ and are pullbacks of representative
functions on the adjoint quotient $\widetilde G/Z$.  Consequently, they
descend as representative functions to every central quotient
$\widetilde G/C$ with $C\leq Z$.
\end{lemma}

\begin{proof}
By irreducibility and Schur's lemma, for every $c\in Z$ there is a unit scalar
$\chi(c)$ such that $\rho(c)=\chi(c)I$.  Therefore
\[
 z_j(cg)=\chi(c)z_j(g),
 \qquad
 \overline{z_j(cg)}=\overline{\chi(c)}\,\overline{z_j(g)}.
\]
Every quantity in \eqref{eq:hopf-coordinates}, and hence $A,P,Q$, is a
polynomial in the phase-balanced products $z_i\overline{z_j}$; all are
therefore $Z$-invariant.

To see directly that the descended functions are representative, observe
that $z_i\overline{z_j}$ is a matrix coefficient of
$\rho\otimes\overline\rho$.  The center acts on this tensor product as
$\chi(c)\overline{\chi(c)}I=I$, so $\rho\otimes\overline\rho$ factors through
$\widetilde G/Z$.  The functions $A,u,v,\tau$ are linear combinations of its
matrix coefficients, while $P$ and $Q$ are matrix coefficients of finite
direct sums of tensor powers thereof.  Thus all are representative functions
on $\widetilde G/Z$, and their pullbacks to every intermediate central
quotient remain representative.
\end{proof}

\begin{corollary}[Every compact simple central form]
\label{cor:simple-central-forms}
Let $H$ be a compact connected Lie group with simple Lie algebra.  Then there
exist $A_H,P_H,Q_H\in\cR(H)$, with $A_H\geq0$ and $A_H\not\equiv0$, such that
for every $m\geq1$,
\[
 \cI_H(P_H^m)=0,
\]
and, for every $s\geq1$,
\[
 \cI_H(Q_H^sP_H^m)
 =c_m\binom{m-1}{s-1}\cI_H(A_H^{4m+s}).
\]
In particular, $H$ does not have the Mathieu property.
\end{corollary}

\begin{proof}
There is a compact connected simply connected simple group $\widetilde G$ and
a subgroup $C\leq Z(\widetilde G)$ such that
$H\simeq\widetilde G/C$.  By Lemma \ref{lem:center-descent}, the triple from
Theorem \ref{thm:simply-connected-simple} descends to $H$.  The quotient map
pushes normalized Haar measure forward to normalized Haar measure, so the
moment identities and positivity are unchanged by Lemma
\ref{lem:haar-pullback}.
\end{proof}

\section{Arbitrary compact connected groups}

The remaining structural step is to exhibit a simple quotient of every
nonabelian compact connected group.

\begin{proposition}[Adjoint simple quotient]\label{prop:adjoint-simple-quotient}
Let $G$ be a nonabelian compact connected Lie group.  Then there is a
continuous surjective homomorphism
\[
 \pi:G\twoheadrightarrow H
\]
onto a compact connected adjoint simple Lie group $H$.
\end{proposition}

\begin{proof}
Let $\mathfrak g$ be the Lie algebra of $G$.  Compactness gives a decomposition
into ideals
\begin{equation}\label{eq:compact-lie-algebra-decomposition}
 \mathfrak g=\mathfrak z\oplus\mathfrak s_1\oplus\cdots\oplus\mathfrak s_r,
\end{equation}
where $\mathfrak z$ is the center and the $\mathfrak s_i$ are compact simple
Lie algebras.  Since $G$ is nonabelian, $r\geq1$.

The adjoint action of the connected group $G$ cannot nontrivially permute the
finite set of simple ideals, so each $\mathfrak s_i$ is $\Ad(G)$-invariant.
Fix $i$ and define
\[
 \pi_i:G\longrightarrow\Aut(\mathfrak s_i),
 \qquad
 \pi_i(g)=\Ad(g)|_{\mathfrak s_i}.
\]
The image is connected, hence lies in the identity component
$\Int(\mathfrak s_i)$ of $\Aut(\mathfrak s_i)$.  The differential of $\pi_i$
is
\[
 \dd\pi_i(X)=\ad(X)|_{\mathfrak s_i}.
\]
Using \eqref{eq:compact-lie-algebra-decomposition}, its image is exactly
$\ad(\mathfrak s_i)$, the Lie algebra of $\Int(\mathfrak s_i)$.  Because
$\pi_i(G)$ is compact, it is closed.  A closed connected subgroup with the
same Lie algebra as the connected group $\Int(\mathfrak s_i)$ is the whole
group.  Thus $\pi_i$ is surjective onto the compact connected adjoint simple
group $H=\Int(\mathfrak s_i)$.
\end{proof}

\begin{proof}[Proof of Theorem \ref{thm:uniform-nonabelian}]
By Proposition \ref{prop:adjoint-simple-quotient}, choose a surjection
$\pi:G\twoheadrightarrow H$ onto an adjoint compact simple group.  Choose
$A_H,P_H,Q_H\in\cR(H)$ as in Corollary
\ref{cor:simple-central-forms}.  Set
\[
 A=A_H\circ\pi,
 \qquad
 P=P_H\circ\pi,
 \qquad
 Q=Q_H\circ\pi.
\]
Lemma \ref{lem:haar-pullback} gives \eqref{eq:uniform-pure} and
\eqref{eq:uniform-marker-tower}, and preserves the pointwise nonnegativity and
nonvanishing of $A$.
\end{proof}

\section{The torus direction and completion of the classification}

Let $T$ be a compact connected abelian Lie group.  Then $T$ is a torus and
its character lattice $\widehat T$ is isomorphic to $\mathbb Z^d$ for some
$d\geq0$.  Under the resulting identification
\[
 \cR(T)\simeq\C[x_1^{\pm1},\ldots,x_d^{\pm1}],
\]
normalized Haar integration is Laurent constant-term extraction.

We use the following theorem of Duistermaat and van der Kallen
\cite{DuistermaatVanderKallen1998}.

\begin{theorem}[Duistermaat--van der Kallen]\label{thm:dvdk}
Let $0\neq f\in\C[x_1^{\pm1},\ldots,x_d^{\pm1}]$.  If
$\CT(f^m)=0$ for every $m\geq1$, then the origin does not belong to the convex
hull of the exponent support of $f$.
\end{theorem}

\begin{corollary}[The torus Mathieu property]\label{cor:torus}
Every torus has the Mathieu property.
\end{corollary}

\begin{proof}
Let $f,h\in\C[x_1^{\pm1},\ldots,x_d^{\pm1}]$ and assume
$\CT(f^m)=0$ for every $m\geq1$.  If $f=0$ or $h=0$, the conclusion is immediate.
Assume $f\neq0$ and $h\neq0$.  By Theorem \ref{thm:dvdk}, a real linear
functional $\ell$ on $\R^d$ strictly separates the finite support of $f$ from
the origin.  Thus there is $\delta>0$ such that
$\ell(\alpha)\geq\delta$ for every exponent $\alpha$ in the support of $f$.
If $C$ is the minimum of $\ell$ on the finite support of $h$, then every
exponent in the support of $hf^m$ has $\ell$-value at least $C+m\delta$.
For all sufficiently large $m$ this is positive, so the zero exponent is not
in the support of $hf^m$.  Hence $\CT(hf^m)=0$ for all sufficiently large
$m$.
\end{proof}

\begin{proof}[Proof of Theorem \ref{thm:classification}]
For compact connected Lie groups, abelianity is equivalent to being a torus.
Corollary \ref{cor:torus} proves that every torus has the Mathieu property.
If $G$ is nonabelian, Theorem \ref{thm:uniform-nonabelian} produces
$P,Q\in\cR(G)$ with $P^m\in\ker\cI_G$ for every $m\geq1$ but
$QP^m\notin\ker\cI_G$ for any $m\geq1$.  Thus $\ker\cI_G$ is not a Mathieu
subspace.  This proves all equivalences.
\end{proof}

\section{The abelian reductions and an explicit transformed witness}

The universal conjectures of M\"uger--Tuset are already false indirectly:
by \cite[Remark~6.7(2)]{MugerTuset2025} they imply the Jacobian conjecture, so
the three-dimensional Jacobian counterexample discussed in the introduction
disproves them by contraposition.  There is also a direct admissible witness
obtained immediately from \cite[Theorem~2.1]{Long2026}.  Write
\[
 f_0(t,w)=(1-w^{-1})\bigl((1-t)+tw\bigr),
 \qquad
 \widetilde f(x,w)=f_0(x^2,w).
\]
The substitution $t=x^2$ and the moment identities in \cite{Long2026} give,
for every $m\geq1$,
\begin{align*}
 \int_0^1\CT_w\!\bigl(\widetilde f(x,w)^m\bigr)x\,\dd x
 &=0,\\
 \int_0^1\CT_w\!\bigl(w^{-1}\widetilde f(x,w)^m\bigr)x\,\dd x
 &=\frac{(-1)^{m-1}}{2(m+1)}\neq0.
\end{align*}
Thus $\widetilde f$, $w^{-1}$, and the admissible weight $\delta(x)=x$
already give a direct counterexample to Conjecture~6.3 at $N=M=1$; the
unchanged $w$-spectrum $\{-1,0,1\}$ also contradicts Conjecture~6.6.

The purpose of this section is different.  We identify the exact image of the
universal Hopf pair under the square-root-free M\"uger--Tuset coordinates for
$SU(2)$.  The resulting witness lies in the distinguished group-coordinate
algebra, uses the actual $SU(2)$ weight $\delta(x)=x$, and retains the full
positive Pascal marker tower.

\begin{proposition}[An explicit $SU(2)$ Laurent witness]
\label{prop:explicit-abelian-SU2}
Let $x\in[0,1]$ and let $w$ be a unit-circle variable.  Define
\[
 U=2x(1-x^2)w,
 \qquad
 V=2xw^{-1},
 \qquad
 T=1-2x^2,
\]
and
\begin{equation}\label{eq:explicit-abelian-pair}
 P_{\mathrm{ab}}=(1+U)\bigl(V-(2+U)T^2\bigr),
 \qquad
 Q_{\mathrm{ab}}=U.
\end{equation}
Then, for every $m\geq1$,
\begin{equation}\label{eq:explicit-abelian-pure}
 2\int_0^1 \CT_w\!\bigl(P_{\mathrm{ab}}^m\bigr)x\,\dd x=0,
\end{equation}
and, for every $s\geq1$,
\begin{equation}\label{eq:explicit-abelian-marker}
 2\int_0^1 \CT_w\!\bigl(Q_{\mathrm{ab}}^sP_{\mathrm{ab}}^m\bigr)x\,\dd x
 =c_m\binom{m-1}{s-1}.
\end{equation}
Moreover,
\begin{align}
 P_{\mathrm{ab}}
 ={}&2xw^{-1}-2(6x^4-6x^2+1)\notag\\
 &+6x(x^2-1)(2x^2-1)^2w\notag\\
 &-4x^2(x^2-1)^2(2x^2-1)^2w^2,
 \label{eq:explicit-abelian-expansion}
\end{align}
so its $w$-spectrum is exactly
\begin{equation}\label{eq:explicit-abelian-spectrum}
 \operatorname{Sp}_w(P_{\mathrm{ab}})=\{-1,0,1,2\}.
\end{equation}
\end{proposition}

\begin{proof}
One has $UV+T^2=1$.  Consequently,
\[
 UP_{\mathrm{ab}}
 =(1+U)\bigl(1-(1+U)^2T^2\bigr),
\]
which is the same defect-one identity as \eqref{eq:no-division-identity}.
The change of variables $t=1-2x^2$ gives
$2x\,\dd x=-\tfrac12\dd t$.  Constant-term extraction in $w$, followed by
evenness in $t$, therefore reduces \eqref{eq:explicit-abelian-pure} and
\eqref{eq:explicit-abelian-marker} exactly to the calculation in Theorem
\ref{thm:hopf-coefficient}.  Direct expansion gives
\eqref{eq:explicit-abelian-expansion}; each of its four displayed coefficient
polynomials is nonzero, proving \eqref{eq:explicit-abelian-spectrum}.

We now identify this pair with the M\"uger--Tuset transform without retaining
any conjugate variables.  Write
\[
 g=\begin{pmatrix}a&b\\ c&d\end{pmatrix}\in SU(2).
\]
Since $\overline a=d$ and $\overline c=-b$ on $SU(2)$, the Hopf quantities of
the first column have the following polynomial representatives in
$\C[a,b,c,d]$:
\begin{equation}\label{eq:polynomial-entry-hopf}
 A_0=ad-bc,
 \qquad U_0=-2ab,
 \qquad V_0=2cd,
 \qquad T_0=ad+bc.
\end{equation}
Thus the restrictions to $SU(2)$ of the universal pair are represented by
\begin{equation}\label{eq:polynomial-entry-pair}
 \widehat P=(A_0+U_0)\bigl(A_0^2V_0-(2A_0+U_0)T_0^2\bigr),
 \qquad
 \widehat Q=U_0,
\end{equation}
which are genuine polynomials in the four matrix entries.

Multiplication by the maximal-torus factor
$\psi(z)=\operatorname{diag}(z,z^{-1})$ sends
\[
 (a,b,c,d)\longmapsto(az,bz^{-1},cz,dz^{-1}).
\]
Every expression in \eqref{eq:polynomial-entry-hopf} is unchanged, so the
variable $z$ cancels.
Under the square-root-free substitution of
\cite[Lemma~5.2]{MugerTuset2025},
\[
 (a,b,c,d)
 =\bigl(iw(1-x^2),\,ix,\,ix,\,-iw^{-1}\bigr),
\]
one obtains
\[
 A_0\longmapsto1,
 \qquad U_0\longmapsto U,
 \qquad V_0\longmapsto V,
 \qquad T_0\longmapsto T.
\]
Hence \eqref{eq:polynomial-entry-pair} maps exactly to
\eqref{eq:explicit-abelian-pair}.  This is the polynomial-entry substitution
required by M\"uger--Tuset: no conjugation variable survives.  In the proof
of Proposition~5.3 of \cite{MugerTuset2025}, they also note that the
substituted polynomial is independent of the chosen polynomial
representative.  After normalizing
the circle integrals as constant terms and Haar measure to have total mass
one, their Proposition~5.3 gives precisely the functional
$2\int_0^1\CT_w(\cdot)x\,\dd x$, in agreement with
\eqref{eq:explicit-abelian-pure} and
\eqref{eq:explicit-abelian-marker}.
\end{proof}

\begin{corollary}[Failure of the universal abelian conjectures]
\label{cor:abelian-reductions}
The universal moment conjecture and the universal convex-support conjecture of
M\"uger--Tuset \cite[Conjectures~6.3 and~6.6]{MugerTuset2025} are false.  The
stronger growth assertion proposed in
\cite[Remark~6.7(1)]{MugerTuset2025} is also false.  More generally, every
auxiliary abelian conjecture in
\cite{Zwart2023,Zwart2024,Zwart2025} that is proved there to imply the Mathieu
conjecture for a specified nonabelian compact connected group is false for
that group.
\end{corollary}

\begin{proof}
Up to the harmless overall factor $2$, the functional in Proposition
\ref{prop:explicit-abelian-SU2} is the M\"uger--Tuset functional
$\int_{(1,1)}(\cdot)\,x$.  Taking
$f=P_{\mathrm{ab}}$, $g=Q_{\mathrm{ab}}$, and $\delta(x)=x$, equations
\eqref{eq:explicit-abelian-pure} and
\eqref{eq:explicit-abelian-marker} with $s=1$ give vanishing pure moments but
nonzero $g$-marked moments for every $m\geq1$.  This directly disproves
Conjecture~6.3.  Equation \eqref{eq:explicit-abelian-spectrum} gives
$0\in\operatorname{Sp}_w(P_{\mathrm{ab}})$, hence directly disproves
Conjecture~6.6.  At the same time,
\[
 \limsup_{m\to\infty}
 \left|2\int_0^1\CT_w\!\bigl(P_{\mathrm{ab}}^m\bigr)x\,\dd x\right|^{1/m}
 =0,
\]
which contradicts the stronger positive-growth assertion in Remark~6.7(1).
The statements in Zwart's papers follow by contraposition from the implication
theorems proved there and Theorem \ref{thm:classification}.
\end{proof}

\begin{remark}[The mixed case $N=M=1$]
For the actual $SU(2)$ reduction, M\"uger--Tuset use $N=1$, $M=2$, and
$\delta(x)=x$; the additional Laurent variable is the maximal-torus variable
$z$.  The pair in Proposition \ref{prop:explicit-abelian-SU2} lies in their
algebra $\mathcal A(SU(2))$ and is independent of $z$.  Deleting this unused
variable places it at $N=M=1$, the smallest case containing both a polynomial
variable and a Laurent variable.  The quadratic reparametrization at the
beginning of this section gives another admissible witness at the same pair
$(N,M)$.  What distinguishes Proposition \ref{prop:explicit-abelian-SU2} is
its origin as the exact transform of the universal Hopf pair and its complete
positive Pascal marker tower.  By \cite[Remark~6.4]{MugerTuset2025}, the case
$N=0$ is the torus theorem of Duistermaat--van der Kallen, while only partial
results were available for $N=1$, $M=0$.
\end{remark}

\section{Further remarks}

\begin{remark}[Strength of the failure]
For every nonabelian compact connected $G$, the failure is uniform in the
exponent: the same fixed multiplier $Q$ detects every positive power of $P$.
The full Pascal marker tower \eqref{eq:uniform-marker-tower} holds on $G$
itself, not merely on a simply connected simple cover.
\end{remark}

\begin{remark}[No classification by Lie type]
The proof uses only one simple root and its highest-weight-one doublet.  It
therefore handles the classical and exceptional Lie types uniformly, as well
as all simply connected, adjoint, and intermediate central forms.  Products,
torus extensions, and finite central identifications are absorbed by the
adjoint-simple-quotient argument.
\end{remark}

\begin{remark}[Scope]
The theorem classifies Mathieu's original compact connected group problem.
It does not by itself classify disconnected compact groups, noncompact groups,
or homogeneous spaces $G/H$ with $H$ nonnormal; those require separate
formulations or additional transfer arguments.
\end{remark}

\section*{AI-assisted research disclosure}

The author used ChatGPT 5.6 Sol Pro (OpenAI) and Claude Opus 5 (Anthropic) for
exploratory calculations, proof organization, literature discovery, and
language editing.  The author reviewed and checked the mathematical
arguments, calculations, citations, and final text and assumes sole
responsibility for the contents.  Neither AI system is an author.

\end{document}